\documentclass[12pt,dvips]{article}
\usepackage[dvips]{graphicx}
\usepackage{color}
\usepackage{subfigure}
\usepackage{rotating}
\usepackage{amsmath}
\usepackage{amsthm}
\usepackage{amsfonts}

\newtheorem{theorem}{Theorem}[section]

\newtheorem{lemma}[theorem]{Lemma}

\newtheorem{corollary}[theorem]{Corollary}

\begin{document}

%


\begin{center} {\Large{\bf On the interpolation property and dominated decomposition property of quasimartingales}}\\

\bigskip

\today

\bigskip

\large {Liang Hong}
\footnote{Liang Hong is an Assistant Professor in the Department of Mathematics, Robert Morris University, 6001 University Boulevard, Moon Township, PA 15108, USA. Tel: 412-397-4024. Email address: hong@rmu.edu.}

\end{center}

\vskip 10pt
\bigskip

\centerline{\noindent {\bf Abstract}} \noindent For a quasimartingale majorized by another quasimartingale,  it is natural to ask whether a third quasimartingale can be inserted between them. In this paper, we give an affirmative answer to this problem. We also establish a dominated decomposition property of quasimartingales. In addition, we show that a weak interpolation property holds for supermartingales and local supermartingales. Our approach also yields the interpolation property and dominated decomposition property for Markov chains.

\bigskip

\bigskip


\vskip 15 pt

\newpage
\section {Introduction}
Quasimartingale is an important class of stochastic processes. It has been extensively studied in the literature of probability. Fisk (1965) first coined the term ``quasimartingale''. Rao (1969) gave the well-known Rao decomposition of quasimartingales. F$\ddot{o}$llmer (1972, 1973) introduced the F$\ddot{o}$llmer measure and showed that a quasimartingale can be represented by a finite signed measures on the $\sigma$-field of predictable sets. Pop-Stojanovic (1972) studied weak Banach-valued quasi-martingales. Letta (1979) proposed a new definition of quasimartingales using stochastic variation. Stricker (1975) gave an integral criterion which is equivalent to the definition of quasimartingales. Jain and Monrad (1982) decomposed the paths of a Gaussian quasimartingale into a martingale and a predictable process of bounded variation. Brooks and Dinculeanu (1988) gave a Doob-Meyer decomposition for Banach-valued quasimartingales.
Vinh (2003) classified quasimartingales in the limit. Brooks et al (2005) gave several decomposition of weak quasimartingales. Cassese (2012) extended the Rao decomposition to quasi-martingales indexed by a linearly order set. The basic quasimartingale theory can be found in Dellacherie and Meyer (1982), He et al (1992), Jacod and Shiryaev (2003) and Protter (2005).


In this paper, we consider the following problem. Suppose $X$ and $Y$ are two quasimartingales and $X\leq Y$. Is it always possible to find a quasimartingale $Z$ such that $X\leq Z\leq Y$? We settle this problem affirmatively. We also give a dominated decomposition property of quasimartingales. Since quasimartingales and supermartingales are closely related through Rao decomposition, we also investigate the same problem for supermartingales. We show that the classes of supermartingales and local supermartingales both satisfy a weak interpolation property. The main tool employed in our work is the theory of Riesz spaces, a branch of functional analysis. The theory of Riesz spaces may lead to surprising elegant solution to seemingly complicated problem. For example, the elegant proof of Lebesgue decomposition theorem in Yosida and Hewitt (1952) is based on the simple fact that every band in a Dedekind complete Riesz space is a projection band. Our work was partially inspired by theirs. There are two obvious advantages of using theory of Riesz spaces to study the interpolation property problem under consideration: (1) It leads to a simple and clean proof. (2) It illustrates the lattice structure of the space under consideration.
H$\ddot{u}$rzeler (1984) extended real-valued quasimartingales to the ones with values in a Banach space. Recent developments along this line can be found in Kuo et al (2004, 2006) and Grobler (2010, 2011). We would like to point out that our work is not in that direction.

The remainder of the paper is organized as follows. Section 2 gives notations and some essential concepts in Riesz spaces. Section 3 documents the main results. Section 4 concludes the paper with a discussion.

\section{Notations and Setup}
Our discussion will always be based on a filtered probability space $(\Omega, \mathcal{F}, (\mathcal{F}_t)_{t\geq 0}, P)$. All relevant mathematical objects such as stopping times, semimartingales, quasimartingales are assumed to be defined on this space. The filtration $(\mathcal{F}_t)_{t\geq 0}$ is assumed to satisfy the usual conditions. We will often omit the reference filtration $(\mathcal{F}_t)_{t\geq 0}$ when no confusion may arise. All vector spaces are real. All the stochastic processes are assumed to be c$\acute{a}$dl$\acute{a}$g, i.e., right-continuous with left-hand limits. We equip all the spaces of stochastic processes with the usual order, that is, for two stochastic processes $X$ and $Y$ we define $X\leq Y$ if and only if $X_t\leq Y_t$ for all $t\geq 0$, where $X_t\leq Y_t$ is understood to be almost surely. Here and throughout, the following notations will be used.
\begin{enumerate}
  \item[]$\mathcal{C}=$ the space of all nonnegative supermartingales.
  \item[]$\mathcal{K}=$ the space of all supermartingale.
  \item[]$\mathcal{Q}=$ the space of all quasimartingales.
\end{enumerate}

Next, we give some basic concepts in the theory of Riesz spaces. For further details, readers may consult Luxemburg and Zaanen (1971), Schaefer (1974), Aliprantis and Burkinshaw (1985, 2000), Zaanen (1997) and Aliprantis and Tourky (2007).

A partially ordered set $X$ is called a \emph{lattice} if the infimum and supremum of any pair of elements in $X$ exist. A nonempty subset $\mathcal{W}$ of a vector space $X$ is called a \emph{wedge} \footnote{Some authors call this a convex cone.} if it is closed under addition and multiplication of nonnegative scalars, that is,
\begin{enumerate}
  \item [(a)]$\mathcal{W}+\mathcal{W}\subset W$,
  \item [(b)]$\lambda \mathcal{W}\subset \mathcal{W}$ for all $\lambda>0$.
\end{enumerate}
A wedge $\mathcal{C}$ is called a \emph{cone}\footnote{Some authors call this a pointed convex cone (with vertex at zero).} if $\mathcal{C}\cap (-\mathcal{C})=\{0\}$, where $-\mathcal{C}=\{-x\mid x\in \mathcal{C}\}$. A cone $\mathcal{C}$ is said to be \emph{generating} in a vector space $X$ if $X=\mathcal{C}-\mathcal{C}$, where $\mathcal{C}-\mathcal{C}=\{x-y\mid x, y\in \mathcal{C}\}$. A real vector space $X$ is called an \emph{ordered vector space} if its vector space structure is compatible with the order structure in a manner such that
\begin{enumerate}
  \item [(a)]if $x\leq y$, then $x+z\leq y+z$ for any $z\in X$;
  \item [(b)]if $x\leq y$, then $\alpha x\leq \alpha y$ for all $\alpha\geq 0$.
\end{enumerate}
If $X$ is an ordered vector space, then $X+=\{x\in X\mid x\geq 0\}$ is called the \emph{positive/standard cone} of $X$. 
An ordered vector space $X$ is said to satisfy the \emph{interpolation property} if for every pair of nonempty finite subsets $E$ and $F$ of $X$ there exists a vector $x\in X$ such that $E\leq x\leq F$.
An ordered vector space is called a \emph{Riesz space} (or a \emph{vector lattice}) if it is also a lattice at the same time. A vector subspace $V$ of a Riesz space $L$ is called a \emph{Riesz subspace} if for any $x, y\in V$ the supremum $x\vee y$ belongs to $V$.

\section{The main results}
\subsection{The interpolation property and dominated decomposition property of quasimartingales}
First, we show that the class of quasimartingales satisfies the interpolation property.
Recall that the \emph{variation} of a process $X$, denoted by Var$(X)$, is defined as
\begin{equation*}
Var(X)=\sup_{\pi}\left[\sum_{i=0}^{n-1} E\left[|X_{t_i}-E[X_{t_{i+1}}\mid \mathcal{F}_{t_i}]|\right]+E|X_{t_n}|\right],
\end{equation*}
where $\pi: 0=t_0< t_1<...<t_n<\infty$ is a finite partition of $[0, \infty)$. A process $X$ is called a \emph{quasimartingale} if $X_t$ is integrable for all $t\geq 0$ and $X$ has finite variation. It seems to be a formidable challenge to show that $\mathcal{Q}$ is a Riesz space using the above definition. Rao (1969) showed that an adapted process is a quasimartingale if and only if $X=X^1-X^2$, where $X^1$ and $X^2$ are both nonnegative supermartingales. This decomposition of quasimartingales is often referred to as \emph{Rao decomposition}. We will use Rao decomposition to show that the space $\mathcal{Q}$ is a Riesz space. To this end, recall that $\mathcal{C}$ represents the space of all nonnegative supermartingales, i.e. ,
\begin{eqnarray*}
\mathcal{C} = \{X \mid X \text{\ is a nonnegative supermartingales}\}.
\end{eqnarray*}
Clearly, $\mathcal{C}$ is closed under algebraic addition and multiplication of nonnegative scalars. Hence, $\mathcal{C}$ is a wedge in $\mathcal{Q}$. However, it is not a cone. 
Moreover, $\mathcal{C}$ is not a vector space. Therefore, we first focus on the vector space $\langle\mathcal{C}\rangle$ generated by $\mathcal{C}$. From Rao decomposition, it is easy to see that $\mathcal{C}$ is a generating wedge in $\mathcal{Q}$, that is,
\begin{eqnarray*}
\mathcal{Q} &=& \mathcal{C}-\mathcal{C}=\langle \mathcal{C}\rangle \\
            &=& \{X-Y\mid \text{ $X$ and $Y$ are nonnegative supermartingales\} }.
\end{eqnarray*}

\begin{lemma}\label{lemma3.1}
$\mathcal{Q}$ is a Riesz space.
\end{lemma}

\begin{proof}
It is trivial that $\mathcal{Q}$ is a vector space and the usual order is a compatible partial order on $\mathcal{Q}$, that is, $\mathcal{Q}$ is an ordered vector space. Thus, we only need to show that $\mathcal{Q}$ is closed under the lattice operation $\wedge$.
To this end, let $X^1-Y^1$ and $X^2-Y^2$ be two quasimartingales, where $X^1, Y^1, X^2$ and $Y^2$ are all nonnegative supermartingales. We have
\begin{eqnarray*}
(X^1-Y^1)\wedge (X^2-Y^2) &=& -Y^2+(X^1-Y^1+Y^2)\wedge X^2 \\
        &=&(X^1+Y^2)\wedge (Y^1+X^2) -(Y^1+Y^2).
\end{eqnarray*}
Clearly, $X^1+Y^2, Y^1+X^2$ and $Y^1+Y^2$ are all nonnegative supermartingales. Hence, $(X^1+Y^2)\wedge(Y^1+X^2)$ is a nonnegative supermartingale too. By the Rao decompositsion, $(X^1-Y^1)\wedge(X^2-Y^2)$ is quaimartingale.
\end{proof}

\begin{theorem}[Interpolation property of quasimartingales]\label{theorem3.1}
$\mathcal{Q}$ satisfies the interpolation property. In particular, for two quasimartingales $X$ and $Y$ with $X\leq Y$, there is a quasimartingale $Z$ such that $X\leq Z\leq Y$.
\end{theorem}
\begin{proof}
Riesz (1940) showed that every Riesz space has the interpolation property. Hence, Theorem \ref{theorem3.1} follows from Lemma \ref{lemma3.1}.
\end{proof}

Riesz (1940) also showed that interpolation property is equivalent to the following dominated decomposition property.

\begin{corollary}[Dominated decomposition property of quasimartingales]\label{corollary3.1}
Let $X, X^1, ..., X^n$ be quasimartingales with $|X|\leq |X^1+...+X^n|$. Then there exist $n$ quasimartingales
$Y^1, ..., Y^n$ such that $|X^k|\leq |Y^k|$ for each $1\leq k\leq n$ and $X=Y^1+...+Y^n$. In particular, if $X$ is nonnegative, then $Y^1, ..., Y^n$ can be chosen to be nonnegative.
\end{corollary}
\noindent \textbf{Remark.} If $\mathcal{C}$ was the standard cone $\mathcal{Q}_+$, then Theorem \ref{theorem3.1} would imply the interpolation property for nonnegative supermartingales. However, this is not the case. Therefore, Theorem \ref{theorem3.1} only guarantees that $Z$ is the difference of two nonnegative supermartingales even if $X$ and $Y$ are both nonnegative supermartingales. We call such a property the \emph{weak interpolation property} of nonnegative supermartingales.
Precisely a subset $S$ of an ordered vector space $X$ is said to satisfy the \emph{weak interpolation property} if for every pair of nonempty finite subsets $E$ and $F$ of $S$ there exists a vector $x\in \langle S\rangle$ such that $E\leq x\leq F$. Indeed, this weak interpolation property holds for supermartingales as we shall see next.\\

\subsection{The weak interpolation property of supermartingales and local supermartingles}
Recall that $\mathcal{K}$ denotes the space of supermartingales, i.e.,
\begin{equation}
\mathcal{K}=\{X\mid X \text{ is a supermartingale}\}. \nonumber
\end{equation}
Clearly, $\mathcal{K}$ is closed under addition and multiplication of nonnegative scalars. However, $\mathcal{K}$ is not a vector space. Therefore, we will look at $\mathcal{L}=\langle \mathcal{K}\rangle$ the vector space generated by $\mathcal{K}$. It is easy see that
\begin{equation*}
\mathcal{L}=\mathcal{K}-\mathcal{K}=\{X^1+X^2\mid X^1 \text{\ is a supermartingale and $X^2$ is a submartingale}\}.
\end{equation*}

\begin{lemma}\label{lemma3.2}
$\mathcal{L}$ is a Riesz space.
\end{lemma}
\begin{proof}
It is ready to verify that $\mathcal{L}$ is an ordered vector space. Thus, it remains to show that $\mathcal{L}$ is closed under the lattice operation $\wedge$. To this end, let $X^1+X^2$ and $Y^1+Y^2$ be two elements in $\mathcal{L}$, that is, $X^1$ and $Y^1$ are supermartingales and $X^2$ and $Y^2$ are submartingales. Notice that
\begin{eqnarray*}
& & (X^1+X^2)\wedge (Y^1+Y^2)\\
&=& X^2+X^1\wedge(Y^1+Y^2-X^2)\\
&=& (X^1-Y^2)\wedge(Y^1-X^2)+(X^2+Y^2).
\end{eqnarray*}
Since $X^1-Y^2$ and $Y^1-X^2$ are both supermartingales by hypothesis, so is $(X^1-Y^2)\wedge(Y^1-X^2)$. By hypothesis, $X^2+Y^2$ is a submartingale. Therefore, $(X^1+X^2)\wedge (Y^1+Y^2)\in\mathcal{L}$. This shows that $\mathcal{L}$ is a Riesz space.
\end{proof}

\begin{theorem}[Weak interpolation property of supermartingales]\label{theorem3.2}
The class of supermartingales satisfies the weak interpolation property, that is, for two supermartingales $X$ and $Y$ with $X\leq Y$, there exists a difference of two supermartingales $Z$ such that $X\leq Z\leq Y$.
\end{theorem}

\begin{corollary}[Dominated decomposition property of differences of supermartingales]\label{corollary3.1}
Let $X, X^1, ..., X^n$ each be a difference of supermartingales with $|X|\leq |X^1+...+X^n|$. Then there exist $n$ differences of supermartingales $Y^1, ..., Y^n$ such that $|X^k|\leq |Y^k|$ for each $1\leq k\leq n$ and $X=Y^1+...+Y^n$. In particular, if $X$ is nonnegative, then $Y^1, ..., Y^n$ can be chosen to be nonnegative.
\end{corollary}

Finally, we show that the class of local supermartingales satisfies the weak interpolation property too. Recall that we say an adapted c$\acute{a}$dl$\acute{a}$g process $M$ is a \emph{local supermartingale} if there exists an increasing sequence of stopping times $(T_n)$ such that $T_n\uparrow \infty$ and $M_{t\wedge T_n}$ is a uniform integrable supermartingale for each $n$. Put
\begin{equation*}
\mathcal{K}_{loc}=\{M \mid M\ \text{is a local supermartingale}\}.
\end{equation*}
Evidently, $\mathcal{K}_{loc}$ is a wedge but not a cone. Thus, we consider the vector space generated by it, i.e.,
\begin{equation*}
\mathcal{L}_{loc}=\mathcal{K}_{loc}-\mathcal{K}_{loc}.
\end{equation*}

\begin{lemma}\label{lemma3.3}
The space $\mathcal{L}_{loc}$ is a Riesz space.
\end{lemma}

\begin{proof}
It is evident that $\mathcal{L}_{loc}$ is an ordered vector space. To see that $\mathcal{L}_{loc}$ is closed under the lattice operation $\wedge$, take two elements $M^1+M^2$ and $N^1+N^2$ in $\mathcal{L}_{loc}$, where $M^1$ and $N^1$ are local supermartingales and $M^2$ and $N^2$ are local submartingales. Then for $i=1, 2$, we can choose increasing sequences of stopping times $(S^i_n)$ and $(T^i_n)$ such that $S^i_n\uparrow \infty$, $T^i_n\uparrow \infty$, and $M^i_{t\wedge S_n}$ and $-N^i_{t\wedge T_n}$ are all uniformly integrable supermartingales for each $n$. Put $U_n= S^1_n\wedge S^2_n\wedge T^1_n\wedge T^2_n$. Then $(U_n)$ is an increasing sequence of stopping times such that $U_n\uparrow \infty$. Also, we have
\begin{equation*}
[(M^1+M^2)\wedge (N^1+N^2)]_{t\wedge U_n}=[(M^1-N^2)\wedge(N^1-M^2)]_{t\wedge U_n}+(M^2+N^2)_{t\wedge U_n}.
\end{equation*}
It follows from Doob's optional sampling theorem that $(M^1-N^2)_{t\wedge U_n}$ is a uniformly integrable supermartingale fore each $n$. Likewise, $(N^1-M^2)_{t\wedge U_n}$ and $(M^2+N^2)_{t\wedge U_n}$ are uniformly integrable supermartingale for each $n$. It follows that $(M^1+M^2)\wedge (N^1+N^2)$ is a local supermartingale. Hence, $\mathcal{L}_{loc}$ is a Riesz space.
\end{proof}

\begin{theorem}[Weak interpolation property of local martingales]\label{theorem3.3}
The class of local martingales satisfies the weak interpolation property, that is, for two local martingales $M$ and $N$ with $M\leq N$, there is a local supermartingale $X$ such that $M\leq X\leq N$.
\end{theorem}

\begin{corollary}[Dominated decomposition property of difference of local supermartingales]\label{corollary3.1}
Let $X, X^1, ..., X^n$ each be a difference of local supermartingales with $|X|\leq |X^1+...+X^n|$. Then there exist $n$ differences of local supermartingales $Y^1, ..., Y^n$ such that $|X^k|\leq |Y^k|$ for each $1\leq k\leq n$ and $X=Y^1+...+Y^n$. In particular, if $X$ is nonnegative, then $Y^1, ..., Y^n$ can be chosen to be nonnegative.
\end{corollary}

\section{Discussion}
In this paper, we establish the interpolation property and dominated decomposition property for quasimartingales. We also give a weak interpolation property for supermartingale and local supermartingales.
Since the space of martingles is evidently not a Riesz space, the machinery employed in this paper does not seem to yield the interpolation property for martingales, supermartingels, local martingales and semimartingales. Of course, the key result we cited from Riesz (1940) is a sufficient condition. Riesz (1940) and Namioka (1957) each gave examples showing that a space may still has the interpolation property without being a Riesz space. Therefore, these are open problems along this line.

Finally, we remark that the interpolation property and dominated decomposition property both hold for Markov chains. For completeness, we document these results here.

\begin{theorem}[Interpolation property of Markov chains]\label{theorem4.1}
The space of Markov chains satisfies the interpolation property. In particular, for two Markov chains $X$ and $Y$ with $X\leq Y$, there is a Markov chain $Z$ such that $X\leq Z\leq Y$.
\end{theorem}
\begin{proof}
It is easy to verify that the space of Markov chains is a Riesz space.
\end{proof}

\begin{corollary}[Dominated decomposition property of Markov chains]\label{corollary4.1}
Let $X, X^1, ..., X^n$ be Markov chains with $|X|\leq |X^1+...+X^n|$. Then there exist $n$ Markov chains
$Y^1, ..., Y^n$ such that $|X^k|\leq |Y^k|$ for each $1\leq k\leq n$ and $X=Y^1+...+Y^n$. In particular, if $X$ is nonnegative, then $Y^1, ..., Y^n$ can be chosen to be nonnegative.
\end{corollary}




\end{document}